\documentclass[12pt]{amsart}
\usepackage{amsmath}
\usepackage{amsxtra}
\usepackage{amscd}
\usepackage{amsthm}
\usepackage{amsfonts}
\usepackage{amssymb}
\usepackage{mathrsfs}
\usepackage[all]{xy}

\newtheorem{thm}{Theorem}[section]
\newtheorem{lemma}[thm]{Lemma}
\newtheorem{prop}[thm]{Proposition}

\def\Z{\Bbb Z}

\def\Q{\Bbb Q}
\def\C{\Bbb C}
\def\N{\Bbb N}

\DeclareMathOperator\aut{Aut} \DeclareMathOperator\inn{Inn}
 
 \DeclareMathOperator\out{Out}

\DeclareMathOperator\M{M} 
 
 \DeclareMathOperator\prob{Prob_\Sigma}
 
\DeclareMathOperator\GL{GL} 
 \DeclareMathOperator\Cay{Cay}
 \DeclareMathOperator\SL{SL}

\DeclareMathOperator\G{G} 
\DeclareMathOperator\trace{trace} \DeclareMathOperator\disc{disc}
\DeclareMathOperator\diag{diag} \DeclareMathOperator\mcg{MCG}
 \DeclareMathOperator\I{I} \DeclareMathOperator\Inne{Inn}

\makeatletter

\newcommand{\Rmnum}[1]{\expandafter\@slowromancap\romannumeral #1@}
\makeatother

\begin{document}

\begin{abstract}  Let $\pi:\aut(F_n)\rightarrow \aut(\Z^n)$ be the epimorphism induced by
the isomorphism $\Z^n \cong F_n/F_n'$ and define
$\mathcal{T}_n:=\ker\pi$. We prove that the subset of
$\mathcal{T}_n$ consists of all non-iwip and all non-hyperbolic
elements is exponentially small.
\end{abstract}

\title{Sieve methods in group theory \Rmnum{3}: $\aut(F_n)$ }


\author{Alexander Lubotzky and Chen Meiri}
\address{Einstein Institute of Mathematics \\Hebrew University\\Jerusalem 90914, Israel}
\email{alexlub@math.huji.ac.il, chen.meiri@mail.huji.ac.il}
\maketitle
\date{\today}

\section{Introduction}

Let $\Gamma$ be a finitely generated group. A subset $\Sigma
\subseteq \Gamma$ is called \emph{admissible} if it is symmetric
(i.e. $\Sigma=\Sigma^{-1}$) and the Cayley graph
$\Cay(\Gamma,\Sigma)$ is not bi-partite. Fix an admissible
generating subset $\Sigma$ of $\Gamma$. If $Z\subseteq \Gamma$ then
the asymptotic behavior of the probability $\prob(w_k\in Z)$ that
the $k^{\text{th}}$-step of a random walk on $\Cay(\Gamma,\Sigma)$
belongs to $Z$ can be used to `measure' the density of $Z$ (the
random walk begins at the identity). In fact,
$$\prob(W_k\in Z):=\frac{|\{(s_1,\ldots,s_k)\in \Sigma^k \mid s_1\cdots s_k \in Z\}|}{|\Sigma|^k}$$
We say that $Z$ is \emph{exponentially small with respect to
$\Sigma$} if there exist constants $c,\alpha>0$ such that
$\prob(w_k\in Z) \le ce^{-\alpha k}$ for all $k \in \N$. The set $Z$
is called \emph{exponentially small} if it is exponentially small
with respect to all admissible generating subsets.

One of the first applications of the large sieve method in group
theory was a result of Rivin \cite{Ri1} and Kowlaski \cite{Ko}. They
proved that the set of non-pseudo-Anosov elements in the Mapping
Class Group, $\mcg$ for short, is exponentially small (see also
\cite{Mah}). Their proof uses the homomorphism form the $\mcg$ to
the symplectic group which is induced by the action on the homology
of the surface. Hence, the proof tells us nothing about the Torelli
group which is the kernel of this homomorphism. Kowlaski asked [Ko,
page 135] if the same result also holds for the Torelli subgroup. An
affirmative answer to this question was given in \cite{MS} and in
\cite{LuMe2}. The main idea in both proofs was to use the action of
the Torelli group on the homologies of double covers of the surface
in order to construct similar homomorphisms from the Torelli group
to symplectic groups.

There is a lot of similarity between the $\mcg$ and the automorphism
group $\aut(F_n)$ of a non-abelian free group $F_n$ of rank $n$. In
particular, there are two possible natural analogue notions to
pseudo-Anosov elements: iwip elements or hyperbolic elements (see
Section \ref{sec dfn} for definitions and \cite{KM} for a discussion
on the analogies). Let $\pi:\aut(F_n)\rightarrow \aut(\Z^n)$ be the
epimorphism induced by the isomorphism $\Z^n \cong F_n/F_n'$. The
subgroup $\mathcal{T}_n:=\ker\pi$ is an analog of the Torelli group
and it is finitely generated for $n \ge 3$ by a result of Magnus,
while for $n=2$ it is just the group of inner automorphisms
$\Inne(F_2)$.

The analogy between $\mcg$ (respectively the Torelli group) and
$\aut(F_n)$ (resp. $\mathcal{T}_n$) suggests that the subset of
$\aut(F_n)$ (resp. of $\mathcal{T}_n$) consisting of all non-iwip
and all non-hyperbolic elements is exponentially small. Rivin and
Kapovich proved that this in indeed the case for $\aut(F_n)$
\cite{Ri2}. In this note we show that the same result also holds for
$\mathcal{T}_n$. The idea of the proof is similar to the one we used
for the Torelli group case: we investigate the actions of
$\mathcal{T}_n$ on the abelizations of finite index subgroups of
$F_n$. For this we use a Theorem of Grunewald and the first author
which analyzes these action \cite{GL}. In fact, the situation for
the automorphism group case is somewhat less `symmetric'  then the
Torelli group case (see footnote 3) and we have to consider also the
action on subgroups of index three (and not just the action on
subgroups of index two which are the analog of double covers). Thus,
unlike the case of the $\mcg$ for which all the representations
studied (in \cite{LuMe2} or \cite{MS}) were naturally defined over
$\Z$, we have to consider here representations onto a subgroup $H$
of $\GL_{n-1}(\Z[\xi])$ where $\xi$ is a non-trivial cubic root of
unity. This brings some new challenges. For example, along the way
we have to prove (see Proposition \ref{xi} below) that the set
$$\{g \in H\mid \exists m \ge 1 \text{ such that the characteristic
polynomial of } g^m \text{ belongs to }\Z[t]\}$$ is exponentially
small. This is proved in Section 2. In Section 3, we describe the
Grunewald-Lubotzky theorem in the form needed here, while in Section
4, we discuss iwip and hyperbolic elements and prove the main result
of this paper- Theorem \ref{main thm}.

\section{Characteristic polynomials}

For a number field $K$ and an element $g \in \GL_n(K)$ let
$f_g:=\det(t\I_d-g)$ denote the characteristic polynomial of $g$ and
let $R_g$ denotes the set of roots of $f_g$. Let $\xi$ be a
non-trivial third root of unity.

Fix a subgroup $\Gamma$ of $\GL_n(\Z[\xi])$ which is commensurable
to $\SL_n(\Z[\xi])$. The goal of this section is to show that that
set
$$\{g \in \Gamma \mid \exists m \ge 1 \text{ such that }\ f_{g^m} \in \Z[t]\}$$
is exponentially small.

Let us recall and set up the notations of the process of
``restriction of scalars''. We can view $\Z[\xi]$ as a free
$\Z$-module of rank 2 with basis $1,\xi$. If $b \in \Z[\xi]$ then
the map $x \mapsto bx$ is a $\Z$-homomorphism of $\Z[\xi]$. Thus, we
have an embedding $\psi:\Z[\xi] \hookrightarrow \M_{2 \times 2}(\Z)$
(the embedding depends on the chosen basis of $\Z[\xi]$). The image
of $\Z$ under this embedding is the set of scalar matrices.
Moreover, an element $x \in \Z[\xi]$ belongs to $\Z$ if and only if
the $(2,1)$-coordinate of $\psi(x)$ equals to zero. We can view
$\M_{2n \times 2n}(\Z)$ as the ring of matrices of size $n \times n$
with entries in $\M_{2 \times 2}(\Z)$. Thus, $\psi$ induces the
restriction of scalars embedding $\varphi:\Gamma \hookrightarrow
\GL_{2n}(\Z)$.\footnote{There exists $h \in \GL_{2n}(Q(\xi))$ and an
automorphism $\alpha$ of the algebraic closure of $\Q$ such that
$\alpha(\xi)=\xi^{-1}$ and $h\varphi(g)h^{-1}=\diag(g,\alpha(g))$
for every $g \in \Gamma$ (where
$\alpha(g)_{i,j}:=\alpha(g_{i,j})$).} Explicitly, if $g \in \Gamma$
then $\varphi(g)_{2(i-1)+k,2(j-1)+l}=\psi(g_{i,j})_{k.l}$ for every
$1 \le i,j \le n$ and every $1 \le k,l \le 2$. In particular, the
trace of an element $g \in \Gamma$ belongs to $\Z$ if and only if
$\sum_{i=1}^{n}\varphi(g)_{2(i-1)+1,2(i-1)+2}=0$. Let $\G(\C)$ be
the Zariski-closure of $\varphi(\Gamma)$ in $\GL_{2n}(\C)$. The
connected component $\G^\circ(\C)$ of $\G(\C)$ is isomorphic to
$\SL_n(\C) \times \SL_n(\C)$ and in particular it is semisimple.

\begin{lemma}\label{128} The subset $Z:=\{g \in \Gamma \mid \trace(g)\in \Z\}$ is
exponentially small in $\Gamma$.
\end{lemma}
\begin{proof}The set $\varphi(Z)$ is contained in the subvariety
$$V(\C):=\{A \in \G(\C) \mid \sum_{i=1}^{n}A_{2(i-1)+1,2(i-1)+2}=0\}$$
where $A_{i,j}$ is the $(i,j)$-coordinate of $A$. It is not hard to
see that $V(\C)$ does not contain any coset of $\G^\circ(\C)$.
Proposition \ref{prop variety} below completes the proof.
\end{proof}

\begin{prop}[see Proposition 5.3 of \cite{LuMe1}]\label{prop variety}
Let $\Gamma$ be a finitely generated subgroup of $\GL_n(\Q)$ such
that connected component $\G^\circ(\C)$ of its Zariski-closure is
semisimple. Assume that $V(\C)$ is a variety defined over $\Z$ and
that $V(\C)$ does not contain any coset of $\G^\circ(\C)$. Then,
$V(\C)\cap \Gamma$ is exponentially small in $\Gamma$.
\end{prop}

Another consequence of Proposition \ref{prop variety} is:
\begin{lemma}\label{129} The subset
$$W:=\{g \in \Gamma \mid \text{ There exists } m\ge 1 \text{ such that } f_{\varphi(g^m)}
\text{ has multiply roots }\}$$ is exponentially small  in $\Gamma$.
\end{lemma}
\begin{proof}
Let $g \in \Gamma$ and assume that the characteristic polynomial of
some positive power of $\varphi(g)$ has multiply roots. Let $m \ge
1$ be the minimal positive integer for which $f_{\varphi(g)^m}$ has
multiply roots. Then, there is a root $x$ of $f_{\varphi(g)}$ and a
primitive $m$-root of unity $\zeta$ such that $\zeta x$ is also a
root of $f_{\varphi(g)}$. Thus, $\zeta$ belong to the normal closure
$K_{\varphi(g)}$ of $\Q(R_{\varphi(g)})$ (recall that
$R_{\varphi(g)}$ is the set of roots of $f_{\varphi(g)}$). However,
$[K_{\varphi(g)}:Q] \le (2n)!$ and there are only finitely many
roots of unity which belong to an algebraic extension of $\Q$ of
bounded degree. So, $m$ is bounded by some constant $c_n$ which
depends only on $n$. In particular, if the characteristic polynomial
of some positive power of $\varphi(g)$ has multiply roots then
$f_{\varphi(g)^{c_n!}}$ has multiply roots. A polynomial has
multiply roots if and only if its discriminant is equal to zero.
Define $W(\C):=\{A \in \G(\C) \mid \disc(f_{A^{c_n!}})=0\}$. It is
not hard to verify that the variety $W(\C)$ does not contain any
cost of $\G^\circ(\C)$. Proposition \ref{prop variety} completes the
proof.
\end{proof}
We are ready to prove the main proposition of this section:
\begin{prop}\label{xi} The set
$$T:=\{g \in \Gamma \mid \exists m \ge 1 \text{ such that } f_{g^m} \in \Z[t]\}$$
is exponentially small in $\Gamma$.
\end{prop}
\begin{proof}
In light of Lemmas \ref{128} and \ref{129} it is enough to show that
$T \subseteq Z \cup W$. For every $g \in \Gamma$ the following
holds:
\begin{itemize}
\item [1.] $R_g \subseteq R_{\varphi(g)}$.
\item[2.] If $\alpha \in \aut(\tilde{Q})$ then $\alpha(R_g) \subseteq R_{\varphi(g)}$ ($\tilde{Q}$ is the algebraic closure of
$\Q$).\footnote{In fact, if $\alpha\in\aut(\tilde{Q})$ and
$\alpha(\xi)\ne \xi$ then $R_g \cup \alpha(R_g)=R_{\varphi(g)}$
while if $\alpha(\xi)=\xi$ then $R_g=R_{\alpha(g)}$.}
\item[3.] $f_g \in \Z[t]$ if an only if $\alpha(R_g)=R_g$ for every $\alpha \in
\aut(\tilde{Q})$.
\end{itemize}

Assume that $g \not \in Z \cup W$. Then as $g \not \in Z$, Condition
3 shows that $\alpha(R_g)\ne R_g$ for some $\alpha \in
\aut(\tilde{Q})$. In turn, conditions 1 and 2 together with the fact
that $g \not \in W$ imply that $\alpha(R_{g^m})\ne R_{g^m}$ for
every $m \ge 1$. The other direction of condition 3 then shows that
$f_{g^m}\not \in \Z[t]$ for every $m\ge 1$.
\end{proof}

\section{Grunewald-Lubotzky Theorem}

Fix $n \ge 3$ and a basis $x_1,\ldots,x_n$ of $F_n$. For $s \ge 2$,
let $K_s$ be the kernel of the homomorphism form $F_n$ to $\Z/s\Z$
which sends $x_n$ to $1$ and $x_1,\ldots,x_{n-1}$ to $0$. Denote
$y_{k,i}:=x_n^{-i}x_kx_n^i$ for $0 \le i \le s-1$ and $1 \le k \le
n-1$. Then, the set $$\{y_{k,i} \mid 0 \le i \le s-1\ \wedge\ 1 \le
k \le n-1\}\cup\{x_n^s\}$$ is a free basis of $K_s$.

There is a natural homomorphism $\alpha:K_s/K_s'\rightarrow
F_n/F_n'$. Since $F_n/K_s$ is abelian every $\varphi \in
\mathcal{T}_n$ preserves $K_s$. Thus, $\varphi$ also acts as an
automorphism on $K_s/K_s'$ and as the identity on $F_n/F_n'$. These
actions commute with $\alpha$, i.e., with a little abuse of notation
we have $\varphi \circ \alpha =\alpha \circ \varphi$. In particular,
$\varphi$ preserves $\ker \alpha$. Denote $$L_s:=\langle
y_{k,i}y_{k,i+1}^{-1}\mid 0 \le i \le s-2\ \wedge\ 1\le k \le
n-1\rangle.$$ Then, $L_s$ is a free factor of $K_s$ and $\ker
\alpha=L_sK_s'/K_s'$.

The abelian group $K_s/K_s'$ has a structure of a $\Z[\xi_s]$-module
where $\xi_s$ be a $s^{\text{th}}$-root of unity. For $k \in K_s$,
$\xi_s(kK_s')=x_n^{-1}kx_nK_s'$. The subgroup $L_sK_s'/K_s'$ is in
fact a free $\Z[\xi_s]$-submodule with a basis
$$D_s:=\{d_{k}\mid   1 \le k \le n-1\}$$ where $d_{k}:=y_{k,0}y_{k,1}^{-1}K_s'$.

Every $\varphi\in\mathcal{T}_n$ acts as an automorphism on
$L_sK_s'/K_s'$ and preserves its structure as a $\Z[\xi_s]$-module.
The group of $\Z[\xi_s]$-automorphisms of $L_sK_s'/K_s'$ is
isomorphic to $\GL_{n-1}(\Z[\xi_s])$ where the isomorphism is depend
on the basis chosen for $L_sK_s'/K_s'$.

Thus, there exists a homomorphism $\rho_s:\mathcal{T}_n \rightarrow
\GL_{n-1}(\Z[\xi_s])$ with respect to the above basis $D_s$.

\begin{thm}[Grunewald-Lubotzky \cite{GL}] Fix $s \ge 2$ and let $\rho_s:\mathcal{T}_n \rightarrow
\GL_{n-1}(\Z[\xi_s])$ be the above homomorphism. Then
$\rho(\mathcal{T}_n)$ is commensurable with $\SL_{n-1}(\Z[\xi_s])$.
\end{thm}

For $s=2$ the above theorem in similar to Proposition 3 of
\cite{LuMe2}. However, we shall see in the next section that unlike
in the Torelli group case where the analog of  $\rho_2$ suffices, in
the $\mathcal{T}_n$ case we also have to consider $\rho_3$.

\section{Iwip and hyperbolic elements}\label{sec dfn}

We start by recalling the definitions of iwip and hyperbolic
elements. A more detailed discussion about these elements and the
analogy to pseudo-Anosov elements in the mapping class group can be
found in \cite{KM}. An element $g \in \aut(F_n)$ is called
\emph{reducible} if there are non-trivial proper subgroups
$H_1,\ldots,H_k < F_n$ such that $H_1*\cdots *H_k$ is a free factor
of $F_n$ and $g(H_{i})$ is conjugate to $H_{i+1}$ for every $1 \le i
\le k$ where the addition in the subscript is modulo $k$. An element
$g \in \aut(F_n)$ is called \emph{irreducible with irreducible
powers}, or \emph{iwip} for short, if for every $m \ge 1$ the
automorphism $g^m$ is not reducible. Hence, if $g \in \aut(F_n)$ is
not iwip then there are $m \ge 1$ and a non-trivial proper free
factor $H$ such that $g^m(H)$ is conjugate to $H$. Rivin \cite{Ri1}
proved that the set of non-iwip elements of $\aut(F_n)$ is
exponentially small.

An element $g \in \aut(F_n)$ is called \emph{hyperbolic} if for
every $m \in \N^+$, the element $g^m$ does not fix any conjugacy
class of a non-trivial element.  Rivin and Kapovich \cite{Ri2}
proved that the set of non-hyperbolic elements of $\aut(F_n)$ is
exponentially small.

Note that both properties, iwip and hyperbolic, are invariant by
multiplication by inner automorphisms, so can be thought (and
usually are thought) as properties of elements of
$\out(F_n)=\aut(F_n)/\inn(F_n)$. As $\mathcal{T}_2=\Inne(F_2)$,
there is no interest in studying these properties in the case $n=2$
and we therefore assume that $n\ge 3$.

The proofs of the above results use the homomorphism
$\pi:\aut(F_n)\rightarrow \aut(\Z^n)$ so they give us no information
on the subgroup $\mathcal{T}_n:=\ker\pi$. However, these results for
$\mathcal{T}_n$ can be obtained by looking at covers.

\begin{prop}\label{iwip} There are homomorphisms
$\rho_1,\ldots,\rho_{2^n-1}:\mathcal{T}_n \rightarrow \GL_{n-1}(\Z)$
and $\psi_1,\ldots,\psi_{3^n-1}:\mathcal{T}_n \rightarrow
\GL_{n-1}(\Z[\xi])$ where $\xi$ is a non-trivial third root of unity
such that:
\begin{itemize}
\item[1.] For every $1 \le i \le 2^n-1$, $\rho_i(\mathcal{T}_n)$ is
of finite index in $\GL_{n-1}(\Z)$.
\item[2.] For every $1 \le i \le 3^n-1$, $\psi_i(\mathcal{T}_n)$
is commensurable with $\SL_{n-1}(\Z[\xi])$.
\item[3.] If $\varphi \in \mathcal{T}_n$ is not iwip then there are
$m \ge 1$, $1 \le i \le 2^n-1$ and $1 \le j \le 3^n-1$ such that the
characteristic polynomial of  $\rho_i(\varphi^m)$ is reducible or
the  characteristic polynomial of  $\psi_j(\varphi^m)$ belongs to
$\Z[t]$.
\end{itemize}
\end{prop}
\begin{proof} We use the notation of the previous section.
Assume that $\varphi\in\mathcal{T}_n$ is not iwip. Then there is a
free basis $x_1,\ldots,x_n$ of $F_n$ and two natural numbers $1 \le
l < n$ and $m \ge 1$ such that $\varphi^m$ preserve the
$F_n$-conjugacy class of $H:=\langle
x_1,\ldots,x_l\rangle$.\footnote{The fact that $\varphi^m$ preserves
$\langle x_1,\ldots,x_l\rangle$ does not imply that it also
preserves $\langle x_{l+1},\ldots,x_n\rangle$ so unlike the Torelli
group case in \cite{LuMe2} we cannot assume that $l<n-1$. This is
the reason why in the current paper we have to consider triple
covers and not only double covers.} Recall that for $s \ge 2$ we
defined $K_s$ to be the kernel of the homomorphism form $F_n$ to
$\Z/s\Z$ which sends $x_n$ to $1$ and $x_1,\ldots,x_{n-1}$ to $0$.
Thus, $K_s$ is a subgroup of index $s$ in $F_n$ so the
$F_n$-conjugacy class of $H$ splits into at most $s$ $K_s$-conjugacy
classes. Hence, $\varphi^{6m}$ preserves the $K_2$-conjugacy class
and the $K_3$-conjugacy class of $H$.

From now on let $s=2 \text{ or }3$ and recall that
$y_{k,i}:=x_n^{-i}x_kx_n^i$ and that $\xi_s$ is a non-trivial
$s$-root of unity. Since $\varphi\in \mathcal{T}_n$ and $F_n/K_s$ is
abelian,  $\varphi(x_n)x_n^{-1}\in K_s$ and
$$\varphi(y_{k,1})K_s'=x_n^{-1}\varphi(y_{k,0})x_nK_s'.$$

Since $\varphi^{6m}$ preserves the $K_s$-conjugacy class of $H$, for
every $1 \le k \le l$ there exists a word $w_k$ such that
$$\varphi^{6m}(y_{k,0})K_s'=w_k(y_{1,0},\ldots,y_{l,0})K_s'.$$
Thus,
$$\varphi^{6m}(y_{k,1})K_s'=w_k(y_{1,1},\ldots,y_{l,1})K_s'$$
and
\begin{equation}\label{eq101}
\varphi^{6m}(y_{k,0}y_{k,1}^{-1})K_s'=w_k(y_{1,0}y_{1,1}^{-1},\ldots,y_{l,0}y_{l,1}^{-1})K_s'.
\end{equation}

Recall that we defined $d_k:=y_{k,0}y_{k,1}^{-1} K_s' $. We also
showed that $K_s/K_s'$ is a $\Z[\xi_s]$-module and that $d_0,\ldots,
d_{n-1}$ freely generates a free $\Z[\xi_s]$-submodule
$L_sK_s'/K_s'$ of $K_s/K_s'$. Equation \ref{eq101} shows that
$\varphi^{6m}$ preserves the $\Z$-submodule generated by
$d_1,\ldots,d_l$. By the definition of $\rho_s$, the image
$\rho_s(\varphi^{6m}) \in \GL_{n-1}(\Z[\xi_s])$ represents the
action of $\varphi^{6m}$ on $L_sK_s'/K_s'$. This implies that at
least one of the following statements holds:
\begin{itemize}
\item $l<n-1$ and the characteristic
polynomial of $\rho_2(\varphi^{6m})$ is reducible.
\item $l=n-1$
and $\rho_3(\varphi^{6m})$ belongs to $\GL_{n-1}(\Z)$. In
particular, the characteristic polynomial of $\rho_3(\varphi^{6m})$
belongs to $\Z[t]$.\footnote{For every $\psi\in \mathcal{T}_n$,
$\rho_2(\psi)\in \GL_{n-1}(\Z)$ so we do not gain any new
information on $\rho_2(\varphi^{6m})$ if $l=n-1$.}
\end{itemize}

There are only $s^n-1$ normal subgroups of $F_n$ of index $s$. The
homomorphism $\rho_s$ depends on the subgroup $K_s$ and on the free
basis $d_1,\ldots,d_{n-1}$ of $L_sK_s'/K_s'$. However, the
characteristic polynomial of $\rho_s(\varphi^{6m})$ does not depend
on the choice of the basis so it is enough to take for each subgroup
of index $s$ just one homomorphism.
\end{proof}

In order to get a similar result for non-hyperbolic elements we will
need the following theorem:

\begin{thm}[Bestvina-Handel \cite{BH}]\label{BHT} If $\varphi \in \aut(F_n)$ is iwip but not hyperbolic then
there is $m \ge 1$ such that $\varphi^m$ is induced by an
automorphism of a compact surface $M$ with one boundary component
$S$.
\end{thm}

\begin{prop}\label{hyperboilc} There are homomorphisms
$\rho_1,\ldots,\rho_{2^n-1}:\mathcal{T}_n \rightarrow \GL_{n-1}(\Z)$
such that:
\begin{itemize}
\item[1.] For every $1 \le i \le 2^n-1$, $\rho_i(\mathcal{T}_n)$ is
of finite index in $\GL_{n-1}(\Z)$.
\item[2.] If $\varphi \in \mathcal{T}_n$ is iwip but not hyperbolic then there are
$m \ge 1$ and $1 \le i \le 2^n-1$ such that the characteristic
polynomial of  $\rho_i(\varphi^m)$ is reducible.
\end{itemize}
\end{prop}
\begin{proof}
We use the notation of the previous section. Let $\varphi \in
\mathcal{T}_n$ be iwip but not hyperbolic. Theorem \ref{BHT} implies
that for some $m \ge 1$, the automorphism $\varphi^m$ is induced by
an automorphism of a compact surface $M$ with one boundary component
$S$.  Thus, $\varphi^m$ sends the homotopic class of $S$ to itself
or to its inverse, so $\varphi^{2m}$ sends the homotopic class of
$S$ to itself. We divide the proof into two cases.
\\
{\bf First case: $M$ is orientable.} In that case $n$ is even and
there exists a free basis $x_1,\ldots,x_n$ of $F_{n}$ such that
$\varphi^{2m}\in \mathcal{T}_n$ preserves the $F_n$-conjugacy class
of $[x_1,x_2]\cdots[x_{n-1},x_n]$. Then, $\varphi^{4m}\in
\mathcal{T}_n$ preserves $K_2$-conjugacy class of
$[x_1,x_2]\cdots[x_{n-1},x_n]$. In particular as
$x_1,\ldots,x_{n-2}\in K_2$ and $d_{n-1}=[x_{n-1},x_n]K_2'$,
$\rho_2(\varphi^{4m})(d_{n-1})=d_{n-1}$. Thus, the characteristic
polynomial of $\rho_2(\varphi^{4m})$ is reducible.
\\
{\bf Second case: $M$ is not orientable.} There exists a free basis
$x_1,\ldots,x_n$ of $F_{n}$ such that $\varphi^{2m}\in
\mathcal{T}_n$ preserves the $F_n$-conjugacy class of of
$x_1^2\cdots x_n^2$. Then, $\varphi^{4m}\in \mathcal{T}_n$ preserves
the $K_2$-conjugacy class of $x_1^2\cdots x_n^2$. It also preserves
the $K_2$-conjugacy class of $x_n^{-1}(x_1^2\cdots x_n^2)x_n$. This
implies that $\rho_2(\varphi^{4m})(d^2)=d^2$ where $d:=d_1\cdots
d_{n-1}$, so the characteristic polynomial of $\rho_2(\varphi^{4m})$
is reducible.

As in the proof of Proposition \ref{iwip} the irreducibility of
$\rho_2(\varphi^{4m})$ depends only on the subgroup $K_2$ and not on
the specific basis of $F_n$. Thus, the number of required
homomorphisms follows from the fact that the are $2^n-1$ subgroups
of index 2.
\end{proof}
By using the large sieve method, Rivin proved:
\begin{prop}[Rivin, \cite{Ri1}]\label{Rir} Fix $n \ge 2$ and let $\Gamma$ be a
subgroup of finite index in $\GL_n(\Z)$. For every $g \in \Gamma$
let $f_g$ be the characteristic polynomial of $g$. Then, the set
$$\{g \in \Gamma \mid \exists m \ge 1 \text{ such that } \ f_{g^m} \text{is reducible}\}$$
is exponentially small.
\end{prop}

We can now conclude:
\begin{thm}\label{main thm} Let $n \ge 3$ and denote $\mathcal{T}_n:=\ker (\aut(F_n)\rightarrow \aut(\Z_n))$. Then the set $Z$ consisting
the elements of $\mathcal{T}_n$ which are either non-iwip or
non-hyperbolic is exponentially small.
\end{thm}
\begin{proof} This follows from Propositions \ref{iwip} and
 \ref{hyperboilc}, using Propositions \ref{xi} and \ref{Rir}.
\end{proof}

\end{document}